\documentclass[11pt]{amsart}

\usepackage{bigints}
\usepackage{amsmath}
\usepackage{centernot}
\usepackage{xcolor}
\usepackage{wrapfig}
\usepackage{tikz-cd}
\usepackage{pgfplots}
\usepackage{enumitem}
\usepackage{hyperref}

\newcounter{cnt1}
\newcounter{cnt2}
\newcommand{\blr}{\begin{list}{$(\roman{cnt1})$}
		{\usecounter{cnt1} \setlength{\topsep}{0pt}
			\setlength{\itemsep}{0pt}}}
	\newcommand{\bla}{\begin{list}{$($\alph{cnt2}$)$}
			{\usecounter{cnt2} \setlength{\topsep}{0pt}
				\setlength{\itemsep}{0pt}}}
		\newcommand{\el}{\end{list}}
	\newtheorem{thm}{Theorem}
	
	\newtheorem{cor}[thm]{Corollary}
	\newtheorem{ex}[thm]{Example}
	
	{\newtheorem{Def}[thm]{Definition}
		\newtheorem{prop}[thm]{Proposition}
		\newtheorem{rem}[thm]{Remark}
		\newcommand{\Rem}{\begin{rem} \rm}
			\newcommand{\bdfn}{\begin{Def} \rm}
				\newcommand{\edfn}{\end{Def}}

\begin{document}
				\large
				\title[The Existence of Linear Selection and the Quotient Lifting Property  ]{The  Existence of Linear Selection and the Quotient Lifting Property }
				
				\author[Monika .]{Monika .}
				\author[Fernanda Botelho]{Fernanda Botelho}
				\author[Richard Fleming]{Richard Fleming}
				\address{Department of Mathematical Sciences \\ The University of Memphis \\ Memphis, TN 38152-3240, United States} 
				
				\email{myadav@memphis.edu, mbotelho@memphis.edu}
				\address{ Central Michigan University    }
				\email{flemi1rj@cmich.edu  }

\maketitle
\vspace{-1cm}
\begin{abstract} Lifting properties for Banach spaces are studied. An alternate version of the lifting property due to  Lindenstrass and  Tzafriri is proposed and a characterization, up to isomorphism, is given. The quotient lifting property for pairs of Banach spaces $(X,J)$, with $J$ proximinal in $X$,  is considered and several conditions for the property to hold are given. \end{abstract}
\vspace{-5mm}
\section{Introduction}

 Following Lindenstrass and Tzafriri \cite{LT},  we say a Banach space $Y$ has the \emph{lifting property} (LP) if for every bounded operator $\psi$ from a Banach space $X$ onto a Banach space $W$ and for every $S\in \L(Y,W)$, there is $\hat{S} \in \L(Y,X)$ such that $S=\psi \circ \hat{S}$. 	We note that in \cite{LT} it is shown that $\ell_1$ has the LP, but that can easily be extended  to show that any space isomorphic to $\ell_1$ has the LP.  The converse for infinite dimensional separable spaces is given in \cite{LT}. 
   If we put the attention on $W$ rather than $Y$, we say $W$ has the \emph{alternate lifting property} (ALP), if for an operator $\psi$ from $X$ onto $W$ and $S\in \L(Y,W)$, there is $\hat{S} \in \L(Y,X)$ such that $S = \psi \circ \hat{S}$. From the ideas in \cite{LT} we  have the following theorem.\\
 
{\bf{Theorem.}} {\it Let $W$ be a Banach space. Then $W$ satisfies the ALP if and only if $W$ is isomorphic to $\ell_1(\Gamma)$, for a suitable  index set $\Gamma$.
	}\\

We  consider the case where $W$ is the quotient of $X$ by a closed subspace $J$ with $\psi = \pi$, where $\pi$ is  the quotient map.  We are interested in the existence of norm preserving lifts of a bounded  operator $S: Y \rightarrow X/J$. If, for every $Y$ and $S$,  such a lift exists then we say that the pair $(X, J)$ has  \textit{quotient lifting property} (QLP).
			
First, we give an alternate version of the QLP.  We say that $Y$ satisfies the \textit{alternate quotient lifting property} (AQLP)  if given a Banach space $X$ with closed subspace $J$  and bounded operator $S$ from $Y$ to $X/J$, there exists $\hat{S}\in \L(Y,X)$ such that $S = \pi\circ \hat{S}$.  We obtain a similar characterization for Banach spaces with the AQLP.\\

{\bf{Theorem.}} {\it If $Y$ is  Banach space then $Y$ satisfies the AQLP if and only if $Y$ is isomorphic to $\ell_1(\Gamma ),$ for a suitable  index set $\Gamma $.\\
	}
	
	We remark that we have not required $\|\hat{S}\| = \|S\|$  in the alternate versions as we do for the QLP. 
		
	In this paper we first study conditions on a Banach space that ensure the existence of  liftings for operators either into the space (ALP) or on the space (AQLP), and for which a diagram commutes. We also investigate conditions for the existence of norm  preserving lifts of operators  (QLP) and its  interconnections with  metric projections, linear selections and  proximinality.\par

We start by recalling some definitions to be used throughout.  For a closed subspace $J$ of $X$, we denote by  $\pi$ the canonical quotient map from $X$ to $X/J$, given by $\pi (x) =x+J$ ($\|\pi\|=1$ unless $J=X$). A subspace $J$ of a normed linear space $X$ is called \textit{proximinal} (resp. \textit{Chebyshev}) if for each $x$ in $X$, the set of best approximations to $x$ in $J$,
\begin{center}
	$P_J(x) := \{j \in J| \|x-j\|= dist(x,J)\}$
\end{center}
is nonempty (resp. a singleton). The set valued map $P_J$ is called the \textit{metric projection onto $J$}. For a proximinal subspace $J$, a \textit{selection} for $P_J $ is a function $p: X\to J$ with $p(x) \in P_J(x)$, for every $x$. If $p$ is linear we call it a \textit{linear selection}. The \textit{metric complement} of $J$ is defined to be $J_0=\{x \in X: \|x\|=dist(x, J)\}.$ We also recall that an M-ideal in a Banach space is a subspace for which the annihilator is an L-summand of the dual   space. For details we refer the reader to  Chapter 1  in \cite{hww}. It is a known fact that M-ideals are proximinal. 

In section 2 we prove  the theorems stated above. These theorems characterize spaces with the ALP and spaces with the AQLP, up to  isomorphism. We  draw some conclusions concerning subspaces  of Banach spaces with the  property. \par
In section 3  we show  that proximinality is a necessary condition for the QLP  and the metric complement being a subspace is sufficient for QLP.    \par
 In section 4  we study the relation between properties of the  metric complement of  $J$  and the   QLP for $(X,J)$. It is worth to mention  that the existence of a linear selection is an invariant condition under isometric isomorphisms.  From this fact we derive that  the QLP holds for several  pairs of  Banach spaces. Moreover,  for proximinal  subspaces, we show that the QLP is equivalent to the existence of a linear selection. \par
 In Section 5 we consider the QLP for  subspaces that are M-ideals. If $X$ is reflexive and $J$ is an $M$-ideal in $X$, then $(X,J)$ has the QLP. The same holds for $C([0,1]^n)$ and $J$ an $M$-ideal in $C([0,1]^n)$. The QLP does not hold, in general, for $(C(\Omega), J)$, with $\Omega$  a compact Hausdorff space and  $J$ an $M$-ideal  of  $C(\Omega )$. Nevertheless,  the QLP holds when  $\Omega$  is compact and metrizable.

\section{The Alternate  Lifting Properties}	

We define the following alternate lifting properties for Banach spaces.
\begin{Def} Let $W$ be a Banach space. Then,
\begin{itemize}
\item \label{ALP} $W$ has the ALP if and only if given  Banach spaces $X$ and $Y$,  an operator $\psi$ from $X$ onto $W$ and $S\in \L(Y,W)$, there is $\hat{S} \in \L(Y,X)$ such that $S = \psi \circ \hat{S}$.
\item \label{AQLP}$Y$ has the AQLP  if and only if for every  Banach space $X$, a closed subspace $J$  and a bounded operator $S$ from $Y$ to $X/J$, there exists $\hat{S}\in \L(Y,X)$ such that $S = \pi\circ \hat{S}$, with $\pi$ denoting the quotient map from $X$ onto $X/J$. 
\end{itemize}
\end{Def}

	\begin{thm} \label{ALP-up-to-isomorphism} Let  $W$ be a Banach space. Then $W$ satisfies the  ALP if and only if $W$ is isomorphic to $\ell_1(\Gamma )$, for some suitable index set $\Gamma $.
	\end{thm}	
	\begin{proof} The statement is straightforward for finite dimensional spaces. We present the proof for the  infinite dimensional case. We recall that any Banach space is the quotient of  $\ell_1(\Gamma)$ for some suitable index set $\Gamma$ (see \cite{R}  page 21), and therefore is the range of a bounded operator on $\ell_1(\Gamma)$. Let $W$ satisfy  the ALP and in that definition, let $X = \ell_1(\Gamma )$, $\psi$ a bounded operator from $X$ onto $W$,  $Y = W$, and $S$ the identity operator on $W$, denoted by $Id$.   Then by the  ALP, there exists $\hat{S}$ from $W$ to $\ell_1(\Gamma)$ such that 
	\[Id = S = \psi \circ \hat{S}.\]
	It is easy to see that $\hat{S}\circ \psi$ is a bounded projection onto a subspace of $\ell_1(\Gamma)$.  It  follows from the Koethe extension of Theorem 2.a.3 in \cite[p. 108]{LT} that W is isomorphic to $\ell_1 (\Gamma )$. 
	 
	 On the other hand, suppose $T$ is an isomorphism from $\ell_1(\Gamma )$ onto $W$.  Let $w_{\gamma} = Te_{\gamma}$, where $\{e_{\gamma}\}$ denotes the standard family  of functions from $\Gamma$ into the scalar field such that  $e_{\gamma} (\alpha)= 1$ for $\gamma = \alpha,$ and $0$ otherwise. Thus if $w\in W$, there exists an absolutely summable family $\{\alpha_{\gamma}\}$ of scalars such that $w= \sum_{\gamma \in \Gamma} \alpha_{\gamma} w_{\gamma}$.  There exists $x_{\gamma} \in X$ such that $\psi(x_{\gamma}) = w_{\gamma}$ and since $\psi$ is surjective and open, the family $\{x_{\gamma}\}$ must be bounded.  For $y\in Y$ there is an absolutely summable family  of scalars $\{\alpha_{\gamma}\}$ such that $S(y) = \sum_{\gamma \in \Gamma} \alpha_{\gamma} w_{\gamma}$.  Since the series $\sum_{\gamma \in \Gamma} \alpha_{\gamma}  x_{\gamma}$ is absolutely summable, it must converge to some $x\in X$.  We define $\hat{S} (y) = x$ and it follows that $\hat{S}$ is a bounded linear operator from $Y$ to $X$ such that $S = \psi \circ \hat{S}$.  
	\end{proof}

Similar considerations apply to  the  AQLP  (see Definition \ref{AQLP})  to   prove the  characterization for  Banach spaces with the AQLP.  
		\begin{thm} \label{a}  Let $Y$ be a  Banach space.  Then $Y$ satisfies the AQLP if and only if $Y$ is isomorphic to $\ell_1(\Gamma )$.
		\end{thm}
		\begin{proof}  Suppose $Y$ has AQLP.  As above, $Y$ is isometric with a quotient space $\ell_1(\Gamma)/J$ for some index set $\Gamma$.  Let $X = \ell_1(\Gamma)$ and $J$ its subspace as given.  Let $S$ be the isometry between $Y$ and $\ell_1(\Gamma)$.  For $\hat{S}$ guaranteed by AQLP,  the operator $\hat{S} \circ S^{-1} \circ \pi$ is a projection onto a complemented subspace of $\ell_1(\Gamma)$.  The remainder of the proof follows as in the proof of Theorem \ref{ALP-up-to-isomorphism}.  
		\end{proof}
		
		Therefore the class of Banach spaces with the ALP coincides with that of the AQLP. 
 
\begin{thm}\label{b} Let $W$ be a Banach space and let $W_1$ be a  complemented subspace of $W$. If $W$ has the AQLP  then $W_1$ has the AQLP. 
	\begin{proof} Let $X$ be a Banach space and $J$ be a closed subspace of $X$.
		Since $W_1$ is complemented in $W$, there exists a projection $P$ from $W$ onto $W_1$. Given $S\in L(W_1, X/J)$ and since $W$ has the AQLP, there exists $\widetilde{S\circ P}$, a lift of $S\circ P$, such that $\pi \circ \widetilde{S\circ P}= S\circ P$. Then $\widetilde{S\circ P}|_{W_1}$ is a lift for $S$. This concludes the proof. 
		\end{proof}
\end{thm}
Theorems  \ref{a}  and \ref{b} imply  that every separable infinite dimensional complemented subspace of a space with the AQLP (or ALP) is isomorphic to $\ell_1$.  We invoke the main theorem in \cite{PC} to conclude that  $C(\Omega, E)$ (with $\Omega$ an infinite compact  Hausdorff space and $E$ an infinite dimensional Banach space) does not satisfy the AQLP. In particular,  $C(\Omega, C(\Omega))$; $C(K_1 \times K_2)$,  where $K_1$ and $K_2$ are infinite compact Hausdorff spaces, do not satisfy the AQLP, since each of these spaces contain a complemented copy of $c_0$.  See \cite{PC}.

\section{Proximinality and the Quotient  Lifting Property}
In this section we  start with  the definition for a pair of Banach spaces to have the QLP. The authors are grateful to T.S.S.R.K.Rao for mentioning a property considered in \cite{hww} that lead to this definition.
\begin{Def} \label{(Quotient Lifting Property)} Let $X$ be a Banach space and $J$ be a closed subspace of  $X$. The pair $(X,J)$ has the QLP if and only if for every Banach space $Y$ and every  bounded operator  $S: Y \to X/J$ there exists  a bounded operator $T$ from $Y$ to $X$ lifting $S$ while preserving the norm, i.e. $\|T\|=\|S\|$ and $\pi \circ T = S$.
	\end{Def}
Given $S$, as in the Definition \ref{(Quotient Lifting Property)},  there may exist several  liftings, i.e. bounded operators $T$ such that $S=\pi \circ T$.  For each such lifting, $T$, we have  $\|S\|\leq \|T\|.$  It is easy to construct examples where the inequality is strict. 


 \begin{thm} \label{proximinality-as a necessary for qlp} Let $X$ be a Banach space and $J$ a closed subspace of $X$. Then
 	\begin{enumerate}
 	 \item[(i)]  If $(X,J)$ has the QLP then $J$ is proximinal in $X$.
 	 \item[(ii)]\label{t4}  If $J$ is proximinal in $X$ and $P_J$ has a linear selection, then $(X,J)$ has the QLP.
 \end{enumerate}
\end{thm}
\begin{proof}
If $J$ is not proximinal in $X$, then there must exist a norm one element $x\in X$ such that $dist(x, J) < \|x-j\|$, for every $j \in J$. 
 We define $Y=span \{x\}$. Let $S: Y \rightarrow X/J$ be such that $S(x) = x+J$ then $\|S\|= dist(x, J)$. Every bounded operator  $T: Y \rightarrow X$ such that $\pi \circ T=S$,  satisfies $T(x) = x+j$, for some $j \in J$. Then $\|S\|<\|x+j\|\leq \|T\|.$  This proves (i).
 \par
For (ii), suppose $P_J$ admits a linear selection $p:X\to J$ and let $\psi : X/J \to X$ be defined by $x + J\mapsto x-p(x)$.  To see that  $\psi$ is well defined, we suppose that  $x_1 +J= x_2+J$. Then $x_2 = x_1+j$ for some $j\in J$, and by the linearity of $p$, we can write \[\psi(x_2 + J) = x_2-p(x_2) = x_1+j-p(x_1)-j=x_1-p(x_1)=\psi(x_1 + J) .\]  Thus $\psi$ is a linear isometry. For a bounded operator $S: Y \to X/J$,  $\tilde{S} = \psi\circ S$ provides the lifting as desired.
\end{proof}
 \begin{cor}
If $ J$  is an M-summand in X, then (X,J) has QLP .
\end{cor}
\begin{proof}

If $ J$  is an M-summand in X, then J is proximinal and has a linear (continuous) selection, namely, the M-projection P.  Therefore the statement.
\end{proof}

 David Yost in \cite{DY} proved that a Banach space is reflexive if and only if  every closed subspace  is proximinal. Hence every nonreflexive space must contain a nonproximinal  closed subspace and therefore this  pair does not have the QLP.
 
 We now give an example of a pair of spaces for which the proximinality condition holds but the QLP does not hold.  First, we  recall the notion of total subset of a Banach space. A subset $F$ of the dual space $X^{*}$ of a Banach space $X$ is said to be total if $f(x) =0$ for each $f \in F$ implies that $x=0$. 
\begin{ex}\label{5} We consider the pair  $(\ell_{\infty},c_0)$. From \cite{hww} we know $c_0$ is an $M$-ideal, hence a proximinal subspace of $\ell_{\infty}$. 
We show that the pair $(\ell_{\infty},c_0)$ does not have the QLP. To see this,  we observe that there is no injective bounded linear map from $\ell_{\infty} / c_0$ to  $\ell_{\infty}$.  
We assume the existence of   an injective bounded linear map $\phi$ from $\ell^{\infty} / c_0$ to  $\ell_{\infty}$.
 Since $\left( \ell_{\infty}\right)^*$ contains a countable total set,  $\{\tau_i : \tau_i(e_j)=\delta_{ij}\}$, 
  then $\phi^* \tau_i$ is a total set for $(\ell_{\infty}/c_0)^*$. Indeed, for $z\in \ell_{\infty}$ such that  $\phi^* \tau_i (z)=0$ for all i, then $\phi z=0$.  Since $\phi$ is injective, then  $z=0$.
   Arterbaum and Whitley in \cite{W} showed that $(\ell_{\infty}/c_0)^*$ has no countable total subset. 
   Therefore, the  identity map from $\ell_{\infty} / c_0$ to $\ell_{\infty} / c_0$ has no lift.
    If such a  lift existed, it would be an injective bounded linear map. This is impossible and it shows that the pair   $(\ell_{\infty} , c_0)$ does not have the QLP.
\end{ex} 
 We note  that the quotient space $\ell_{\infty}/c_0$ does not have  the ALP. To establish this take $S$ to be the  identity map from $\ell_{\infty}/c_0$ to $\ell_{\infty}/c_0$ and $\psi$ to be the quotient map from $\ell_{\infty}$ to $\ell_{\infty}/c_0$. 
 If there exists $\hat{S}$ from $\ell_{\infty}/c_0$ to $\ell_{\infty}$ such that $S=\psi\circ\hat{S}$, then $\hat{S}$  is an injective map but  no such  map exists. Theorem \ref{ALP-up-to-isomorphism}  implies that $\ell_{\infty}/c_0$  is not isomorphic to a $\ell_1(  \Gamma ),$ for any index set  $\Gamma$.

For completeness of exposition we include the following theorems  from  \cite{D} and from \cite{CW} to  be used  later.
\begin{thm}\label{4} (See Deutsch, \cite{D}) Let $J$ be a proximinal subspace of a normed linear space $X$. Then  the following are equivalent
	\begin{enumerate}
		\item $P_J$ has a linear selection.
		\item $J_0,$ the metric complement of $J$, contains a closed subspace $J_1$ such that $X = J\oplus J_1$.
				\end{enumerate}
	\end{thm}

\begin{thm}\label{2}(See Cheney and Wulbert \cite{CW}) Let $J$ be a subspace of $X$. Then
	\begin{enumerate}
		\item $J$ is a proximinal subspace if and only if $X= J+J_0$.
		\item\label{3} $J$ is Chebyshev if and only if $X= J+J_0$ and the representation of each $x \in X$ as $x = j + j_0$, where $j\in J$ and $j_0 \in J_0$, is unique.
	\end{enumerate}
\end{thm}
Thus we have the following proposition.
\begin{prop}\label{J_0 is a subspace}Let $J$ be a  proximinal subspace of $X$. Then
	\begin{enumerate} \item[(i)] If $J_0$ is a subspace of $X$ then $(X,J)$ has the QLP. 
		\item[(ii)] \label{c7}(cf. \cite{hk}) $J_0$ is a subspace if and only if $J$ is Chebyshev and $P_J$ is linear.
	\end{enumerate}
\end{prop}
\begin{proof} From  Theorem \ref{2},  we have $X= J+J_0$. By assumption $J_0$ is a subspace of $X$, so  Theorem \ref{4} asserts that  $P_J$ has a linear selection. So (i) follows from Theorem \ref{t4}-(ii).
\par
 Since  $J$ is proximinal in $X$ and  $J_0$ is a subspace, Theorem \ref{2}-(1) implies  that $X=J+ J_0$. It is clear that $J\cap J_0=\{0\}$, then $X=J\oplus J_0$. The existence of a linear selection of the metric projection $P_J$ follows from Theorem \ref{4}. The uniqueness of the representation of any element in $X$ implies that $J$ is Chebyshev by Theorem \ref{2}-(2). Conversely, since $J_0=\{ x-P_J(x): \, x \in X\}$, the linearity of the metric projection implies $J_0$ is a subspace of $X$. 
\end{proof}
Next example shows that $J_0$ being a subspace is not a necessary condition for the QLP.
\begin{ex}(See example 2.7 of \cite{D})
	Let $X = \ell_{\infty}^{(2)}(\mathbb{R})$ and $J=span \{e_1\}$.   It is easy to see that $J_0$ is not a subspace. Just consider $(0,2)$ and $(1,-2)$, both in $J_0$ with sum  $(1,0) \notin J_0$.  We observe that  $p:X\rightarrow J$ given by  $p(x,y)=(x,0)$ is a linear selection for the metric projection, Theorem 2-(ii) implies that  $(\ell_{\infty}^{(2)}({\mathbb{R}}), J )$ has the QLP. 
\end{ex}

\section{On metric complement of subspaces for low dimensional Banach spaces}
In this section we start with a characterization of the  metric complements of subspaces of $\ell_{\infty}^{(2)}(\mathbb{R})$. First, we notice that for $J=\{0\}$ or $J = X$ then $J_0=X$ or $J_0=\{0\}$ respectively. 

\begin{prop}\label{1} Let $X = \ell^{(2)}_{\infty}(\mathbb{R})$ and let $J$ be a  non-trivial subspace of $X$. Then $J_0$ is a subspace of $X$ if and only if $J$ is generated by $(u,v)$ with $uv\neq 0$.
	\end{prop}
	\begin{proof}
We claim that $J_0$ is homogeneous, i.e. given $(a,b) \in J_0$ and $\lambda$ a nonzero scalar, then $(\lambda a, \lambda b)\in J_0$. Towards this claim we just observe that 
\[ \inf_{j \in J} \|(\lambda a, \lambda b)-j\| =|\lambda | \inf_{j \in J} \|( a,  b)-\lambda^{-1} j\| =|\lambda | \|(a,b)\|= \|(\lambda a, \lambda b)\|. \]

 Let $J$ be a non-trivial subspace of $\ell^{(2)}_{\infty}(\mathbb{R})$ generated by $(u,v)$. If $uv = 0$, WLOG assume that $u\neq 0$ and $v = 0$, then $J=span (1,0)$. If  $J_0$ is a subspace we have   $w_0=(0,1)$ and $w_1=(1,1)$ belong to $J_0$ but  $w_1-w_0=(1,0)$ does not belong to $J_0$. This  shows that  $J_0$ is  not a subspace. If $uv \neq 0$, let $a=\frac{v}{u}$ then $J$ is generated by $(1,a)$. For $x_1 = (1,c)$\vspace{-2mm}
	\begin{center}
	$dist(x_1,J)= \inf_{t \in \mathbb{R}}\{\max \{ |1-t|,|c-at|\}\}$
	\end{center}\vspace{-3mm}

\begin{figure}[h]
	\centering
	\includegraphics[width=1\linewidth]{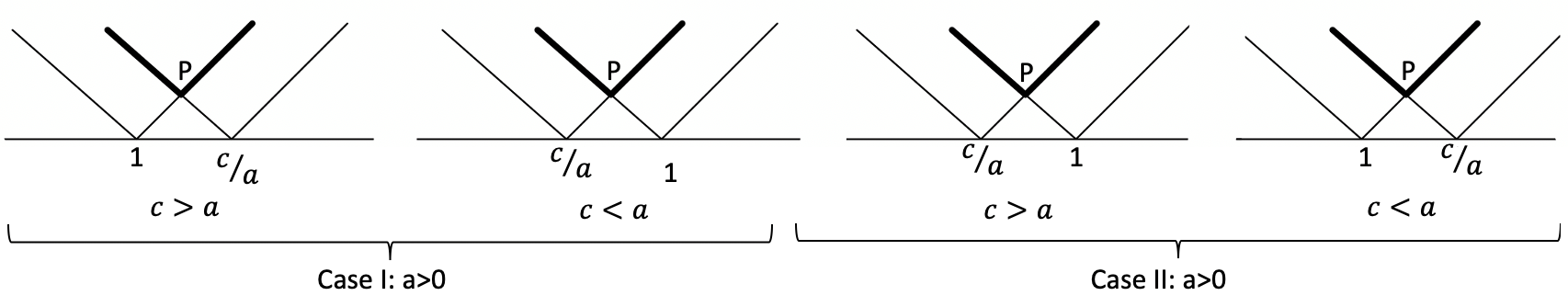}
\end{figure}

 The  graphs of $|1-t|$ and $|c-at|$ plotted in the figure above illustrate  the cases to be considered in the computation of $dist(x_1,J)$. In each case, the distance is equal to the second coordinate of  $P$. We now proceed with the computations. \\
	Case I: $a>0$\\
	If $c>a$, $dist(x_1,J)=\frac{c-a}{1+a}$ which is not equal to $\|x_1\|_\infty$ for any choice of $c$. If $c<a$, $dist(x_1,J)=\frac{a-c}{1+a}$ which is equal to $\|x_1\|_\infty$ only if $c = -1$. So $J_0$ is generated by $(1,-1)$ and is a subspace of $X$.\\
	Case II: $a<0$\\
     If $c>a$, $dist(x_1,J)=\frac{c-a}{1-a}$ which is equal to $\|x_1\|_\infty$ only if $c = 1$. If $c<a$, $dist(x_1,J)=\frac{a-c}{1-a}$ which is not equal to $\|x_1\|_\infty$ for any choice of $c $. So $J_0$ is generated by $(1,1)$ and is a subspace of $X$.
     
	\end{proof}

\begin{rem}\begin{enumerate}
	 \item Let $X$ be a Banach space and $J$  a proximinal subspace of codimension one. Then $(X,J)$ has the QLP. This follows from Corollary 2.8 in  \cite{D}.
	 \item If  $J$ is a closed subspace of a Hilbert space, $\mathcal{H}$, then the pair $(\mathcal{H}, J)$ has the QLP since the metric projection on a subspace of a Hilbert space is linear.
		\end{enumerate}
	\end{rem}

	The next result gives  a sufficient condition for  the existence of a linear selection of  the metric projection onto a proximinal subspace. 
\begin{prop}\label{QLP->LS}
	Let $X$ be a Banach space and let $J$ be a closed  subspace of $X$. If $(X,J)$ has the QLP then the metric projection $P_J$ has a  linear selection.
\end{prop} 
\begin{proof} Since $(X, J)$ has the quotient lifting property then there exists a bounded operator $\tilde{Id} :X/J \rightarrow X$ such that $Id =\pi \circ  \tilde{Id}$ and $\|\tilde{Id}\|=\|Id\|=1$. It is easy to check that $\tilde{Id} \circ \pi$ is a projection on $X$. For simplicity we denote this composition by $P$. Further, $X=Range (P) \oplus Ker (P)$ and $ker (P)=J$. It is clear that $J \subset ker (P)$, if $x \in Ker (P) $ then $\tilde{Id} (x+J)=0$ and $x+J=\pi \circ \tilde{Id} (x+J)=\pi (0)=J$. Hence $x \in J$.   We claim that $Range (P) \subset J_0$, the metric complement of $J$.  Towards this claim, we observe  that,  for $x \in Range (P)$, $P(x) = \tilde{Id} (x+J) =x$ and $\|x\| \leq \|x+J\|= dist (x, J)\leq \|x\|.$ Therefore $x \in J_0$. An application of the Theorem \ref{4} implies that the metric projection onto $J$ has a linear selection.
\end{proof} 

\begin{rem} \label{QP-eq-LS}
Given  $J$, a  proximinal subspace of a Banach space $X$,  and let $P_J$ denote the metric projection onto $J$. Then the theorems  \ref{proximinality-as a necessary for qlp} and \ref{QLP->LS} imply the equivalence of the following statements:  \begin{itemize}
\item $(X,J)$ has the  QLP.
\item   $P_J$ has a linear selection.
\end{itemize}
\end{rem} 
  
{
\begin{ex}

Let $\Omega$ be a compact Hausdorff space  and let $J$  be  the subspace of all  constant  functions in $C(\Omega)$.  The pair   $(C(\Omega), J)$ has  the QLP if and only if the cardinality of $\Omega$ is less or equal to  2. It is clear that $J$ is Chebyshev.

 If $\Omega$ has 3 distinct points, $a, b, c$, we 
consider three pairwise  disjoint open neighborhoods, $U_a$, $U_b$ and $U_c$ of $a$, $b$ and $c$,  respectively. Then there exist continuous functions $f_a$ and $f_b$ on $\Omega$ satisfying the conditions:    $f_a:\Omega \rightarrow [0, 1]$,  $f_a (a)=1$ and $f_a (x)=0,$ for all $x \notin U_a$; $f_b: \Omega \rightarrow  [-1,0]$,  $f_b(b)=-1$ and $f_b(x)=0$, for all $x \notin U_b$.  We set $f=f_a+f_b$. Similarly we define  $g=g_b+g_c$, with $g_b: \Omega \rightarrow  [-1,0]$,  $g_b(b)=-1$ and $g_b(x)=0$,  for all $x \notin U_b$;   $g_c : \Omega \rightarrow   [0,1]$,  $g_c (c)=1$  and $g_c (x)= 0, $ for all $x \notin U_c$. The constant function equal to $0$ is  closest to $f$ and $g$ but $-1/2$ is  closest  to $f+g$. This implies that $P_J$ is not linear and then $(C(\Omega), J)$ does not have the QLP by Remark \ref{QP-eq-LS}. The other implication follows from Proposition \ref{1} .  \end{ex}
}
This example  shows that  $X/J$ having the ALP  does not imply that the metric projection onto $J$ has a linear selection. The lift of the identity operator on $X/J$ may not have norm 1, in which case the argument given in the proof for the  Proposition \ref{QLP->LS} does not hold. 

We add a few remarks that follow straightforwardly from previous results and theorems in \cite{D}. \begin{itemize} 
\item  If $\Omega$ contains n isolated points, there is an n-dimensional subspace $J$ of $C_0(\Omega) $ for which $P_J $ admits a linear selection. Therefore  $(C_0(\Omega),J) $ has the  QLP. 
 \item  If $ J = span\{f\}$, then $(C_0(\Omega),J)$  has the QLP iff the  support of $ f$ contains at most 2 points.
\item If $f$ is in $L^1 (\Omega,\Sigma, \mu) = X,$ then for $ J = span\{f\}$, $(X,J)$ has  the QLP if support of $f$  is purely atomic and contains at most two atoms.  A similar statement holds for $L^p$, with $1<p<\infty$ and $p \neq 2$.   \end{itemize}

\noindent 	The next result addresses the problem of whether the existence  of a linear selection for a pair of spaces  transfers to other pairs of spaces defined from the given one.  We consider sequence spaces, spaces of vector valued continuous functions on a compact Hausdorff space $\Omega$, Lebesgue-Bochner integrable function spaces ($L_1(\mu, X)$) and space of $\mu$-measurable Pettis integrable functions ($\hat{P_1}(\mu, X)$). {	For the definitions and properties of these spaces we refer the reader to \cite{R}.
Given a Banach space $X$ we denote by $\mathcal{S} (X)$ any one of the following sequence spaces: $ \ell_{p}(X) $ ($1\leq p\leq \infty$), $c_0(X)$ or $c(X)$.

We denote by $\mathcal{F} (X)$ any one of the following function spaces: $  C(\Omega,X)$, with a compact Hausdorff space, or  $Lip (\Omega, X)$ , with $\Omega$ a compact metric space endowed with one of the standard norms. We denote by $\mathcal{I} (X)$ any one of the following function spaces:
$L_1(\mu,X)$ or $\hat{P_1}(\mu, X)$. Further, $L(Y,X)$ denotes the space of bounded operators from $Y$ into $X$.} \vspace{-.4mm}
{\begin{prop}\label{gs}
	Let $X$ and $Y$ be  Banach spaces and let $J$be  a proximinal subspace of $X$ such that the metric projection $P_J$ has a linear selection $p: X\to J$. Then the map $f\mapsto p\circ f$ from $\mathcal{S} (X)$ ($\mathcal{F} (X)$,  $\mathcal{I} (X)$ or $ L(Y,X)$) onto $\mathcal{S} (J)$ (resp. $\mathcal{F} (J)$, $\mathcal{I} (J)$,  $\mathcal{L} (Y, J)$) is a linear selection.
\end{prop}} \vspace{-.7mm}
\begin{proof}
		The proof follows easily from $\|(p(x_n))\|\leq 2 \|(x_n)\|.$
		\end{proof}
Let $J$ be a  proximinal subspace of $X$ and  let $p$ be a linear selection. We observe that if $f:Y \rightarrow X$ is an isometric isomorphism, then $f^{-1}(J)$ is proximinal in $Y$, and $f^{-1}\circ p \circ f$ defines a linear selection from $Y$ to $f^{-1}(J).$  
	
 We denote by $X\hat{\otimes}_\pi Y$ the tensor product space of $X$ and $Y$ endowed with the projective norm and by $X\hat{\otimes}_\epsilon Y$ the tensor product space of $X$ and $Y$ endowed with the injective norm. We refer the readers to \cite{R} for more details. 
If $X$ and $Y$ are isometrically isomorphic we denote this by $X\cong Y$. The following identifications of tensor products are known:  $\ell_1\hat{\otimes}_\pi X\cong\ell_1(X)$; $\ell_1\hat{\otimes}_\epsilon  X\cong\ell_1[X]$,  the space of unconditionally summable sequences in $X$; $c_0\hat{\otimes}_\epsilon X\cong c_0(X)$; $C(\Omega)\hat{\otimes}_\epsilon X \cong C(\Omega,X)$, the space of continuous functions from $\Omega$ to $X$; $ L_1(\mu)\hat{\otimes}_\pi X\cong L_1(\mu,X)$ and $L_1(\mu)\hat{\otimes}_\epsilon X \cong \hat{P_1}(\mu,X)$. These facts, together with the above observation and Proposition \ref{gs}, establish the next corollary. 
\begin{cor}\label{tc}  If $X$ is a Banach space, $J$ is a closed proximinal subspace of $X$, and   $p:X\to J$ is a linear selection, then there is a linear selection from $\ell_1\hat{\otimes}_\pi X$ onto $\ell_1\hat{\otimes}_\pi J$, $\ell_1\hat{\otimes}_\epsilon X$ onto $\ell_1\hat{\otimes}_\epsilon J$, $c_0\hat{\otimes}_\epsilon X$ onto $c_0\hat{\otimes}_\epsilon J$, $C(\Omega)\hat{\otimes}_\epsilon X$ onto $C(\Omega)\hat{\otimes}_\epsilon J$, $ L_1(\mu)\hat{\otimes}_\pi X$ onto $ L_1(\mu)\hat{\otimes}_\pi J$, and $L_1(\mu)\hat{\otimes}_\epsilon X $ onto $L_1(\mu)\hat{\otimes}_\epsilon J. $
	\end{cor}
\begin{rem} Using Theorem \ref{t4} we conclude that all the pairs listed  in Proposition \ref{gs} have the QLP.
	\end{rem}
\section{M-ideals and the quotient lifting property}

In this section we explore the QLP for pairs where the subspace is an M-ideal. Towards this we establish the existence of a linear selection for the metric projection.\par

In \cite{hww} (see p.59) it is mentioned that Ando, Choi and Effros proved that if $J$ is an $L_1$-predual M-ideal in Banach space $X$ and $Y$ is a separable Banach space, any bounded linear operator $T\in L(Y,X/J)$ has a bounded linear lifting $\tilde{T}\in L(Y,X)$ such that $\|T\|=\|\tilde{T}\|$ and $T=\pi \circ \tilde{T}$. 
 Our next result shows that this statement holds for $X=C(\Omega)$, where $\Omega$ is compact metrizable, any M-ideal $J$ and all Banach spaces $Y$. Therefore the pair $(C(\Omega), J)$ has the QLP. Let us recall a definition from \cite{Ando}. 
	\begin{Def}
		A Banach space $X$ is called a $\pi$-\textit{space} if there is a sequence ${F_n}$ of finite dimensional subspaces such that $F_1\subset F_2\subset F_3\dots$ with $\overline{\cup_{n=1}^{\infty} F_n}=X$ and each $F_n$ is the range of a projection of norm one.
		\end{Def}
	\begin{rem}\label{rem} Separable $L_p (1\leq p <\infty)$ and $C(\Omega)$ on compact metrizable $\Omega$ are $\pi$-spaces (See \cite{Ando}.)
		\end{rem}
\begin{prop}
The pair $(C(\Omega), J)$, where $\Omega$ is a compact metrizable space and $J$ is an M-ideal of $C(\Omega)$, has the QLP. 
	\end{prop}
\begin{proof}  
	Since $J$ is an M-ideal of $C(\Omega)$, then $J = \{f \in C(\Omega): f|_D \equiv 0 \}$ for some closed subset $D$ of $\Omega$.  By Remark \ref{rem}, $C(\Omega)$  is a $\pi$-space and $C(D)$ is isometrically isomorphic to $C(\Omega)/J$. Therefore $C(\Omega)/J$ is a $\pi$- space. By Theorem 5 of \cite{Ando} there exists a norm one linear map $\phi$ from $C(\Omega)/J$ to $C(\Omega)$ such that $\pi\circ \phi = Id$. For a Banach space $Y$ and bounded linear map $S : Y \to C(\Omega)/J$, $\tilde{S} : Y \to C(\Omega)$ defined by $\tilde{S}= \phi\circ S$ is a lifting such that $S=\pi \circ \tilde{S}$ and $\|S\| =\|\tilde{S}\|$. This concludes the proof.
	\end{proof}

Our next result shows on how to construct a linear selection onto an arbitrary M-ideal of the space of continuous functions on a particular class  of Euclidean domains.

\begin{prop}
	Let $J$ be an M-ideal of $C([0,1]^n)$, where $n$ is a positive integer. Then the metric projection $P_J$ has a linear selection. 
	\end{prop}
\begin{proof}Since $J$ is an M-ideal of $C([0,1]^n)$ there exists a closed subset $D$ of $[0, 1]^n$ such that $J = \{f \in C([0,1]^n): f|_D \equiv 0 \}$ (See \cite{hww} p.4.) 
{	We observe that $g \in J_0$ if and only if $\|g\|_{\infty}= dist (g, J) =\|g|_D\|_{\infty}.$  We present the proof for the cases: $n=1$ and $n=2$; other cases follow similarly. For $n=1$,  we first consider $D=[t_1,t_2]\cup [t_3,t_4].$ Let $f \in C[0,1]$, then we define $f_1$ as
	 \[(\star ) \,\,
	 f_1(x)=\begin{cases} f(x)-f(t_1\chi_{(-\infty,t_1)}+t_4\chi_{(t_4,\infty)}) & \text{for }x\notin [t_2,t_3] \\0 & \text{for } x\in D\\ f(x)-f(t_2)+\dfrac{x-t_2}{t_3-t_2}\left[f(t_2)-f(t_3)\right] &\text{for }x\in[t_2,t_3] \end{cases}
	 \]
	 The map $p: C([0,1])\to J$ defined by $f\mapsto f_1$ is a linear selection and $f-f_1 \in J_0$. For an arbitrary closed subset of $[0,1]$ we assign the value $0$ on $D$ and extend over the  open intervals in $[0,1]\setminus D$ as in $(\star )$.\par
	 \begin{figure}[h]
	 	\centering
	 	\includegraphics[width=0.3\linewidth]{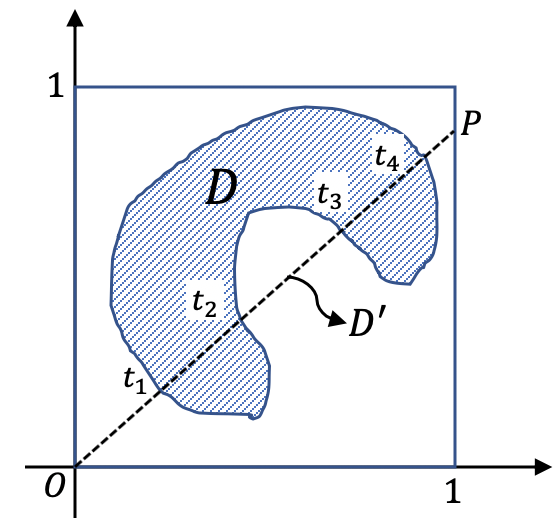}
	 	\caption{$D$ is a closed subset of $[0,1]^2$ and $D'=D\cap \overline{OP}$}
	 	\label{fig:pic-1}
	 \end{figure}
    Let $n=2$ and $f\in C([0,1]^2)$. For $x\in [0,1]^2$ we will define $f_1 = f_{1_m}$ along the line of slope $m$ passing through $0$ and $x$. Let $D'$ be the  intersection of $D$ with the line passing through $0$ and $x$, see Figure 1. For simplicity of notation, we assume that $D'$ has two connected components, i.e.  $D'=[t_1,\, t_2]\cup [t_3,t_4].$ Then, 
      \[
     f_{1_m}(x)=\begin{cases} f_1(x) &\text{for } \|x\|\notin [\|t_2\|,\|t_3\|] \\ f(x)-f(t_2)+\frac{\|x-t_2\|}{\|t_3-t_2\|}[f(t_2)-f(t_3)] &\text{for } \|x\|\in[\|t_2\|,\|t_3\|] \end{cases}
     \]
     The map $p: C([0,1]^2)\to J$ defined by $f\mapsto f_1$ is a linear selection.}
	\end{proof}
The same techniques are  unsuitable  for   general Hausdorff spaces.

It is well known that M-ideals are proximinal (see \cite{hww}, p.50). It is not the case that for every Banach space $X$ and an M-ideal $J$ in $X$,  the pair $(X,J)$ has the QLP,  as we have seen with  the pair $(\ell_{\infty},c_0)$. However the conclusion is different  for reflexive Banach spaces. 

Given a Banach space $B$, let $J_B$ denote the canonical isometric embedding of $B$ into its double dual, $B^{**}$.

\begin{prop}
	If $X$ is reflexive and $J$ is an M-ideal of $X$ then $(X, J)$ has the QLP.
\end{prop}
\begin{proof}
	Let $Y$ be a Banach space and $S:Y \rightarrow X/J$, a bounded operator. We  denote by $\pi :X\rightarrow X/J$, the quotient map, we assume that $J\neq X$.  Since the dual of $X/J$ is identified with the annihilator of $J$, $J^{\perp}$, we have
	$S^*: J^{\perp} \rightarrow Y^*$. Let $P$ denote the projection
	 $  X^*=J^{\perp}\oplus_1 J^* \rightarrow J^{\perp}$. 
	 The composition $S^*\circ P: X^* \rightarrow Y^*$, is such that  $\|S^*\circ P\|=\|S\|$, and  $(S^*\circ P)^*: Y^{**} \rightarrow X^{**}=X$ 
	 also has the same norm as $S$. 
	 We now show that $\pi \circ (J^{-1}_X \circ (S^*\circ P)^* \circ J_Y )=S$.
	   Let $y \in Y$, $(S^*\circ P)^* (y^{**})=  J_X(x)$, and $\pi (x)= x+J$. Given $y \in Y$ 
	  and $\eta \in X^*$ then
	\[ (S^* \circ P)^* \circ J_Y (y) (\eta)= P(\eta) (S(y)).\]
	On the other hand, $S(y)=x+J$, identified with its image in $(X/J)^{**}$,
	$(x+J)^{**} (P(\eta))= P(\eta)(x+J)=P(\eta) (S(y))$. This shows that $\pi \circ (J^{-1}_X \circ (S^*\circ P)^* \circ J_Y )=S$ and completes the proof.
\end{proof}

\vspace{-5mm}
\begin{thebibliography}{99}

	\bibitem {W}D. Arterburn, R. Whitley, {\em Projections in the space of bounded linear operators}, Pacific Journal of Mathematics, \textbf{15:3}
	(1965), 739‐746.
	
	\bibitem{Ando} T. Ando, {\em Closed range theorem for convex sets and linear liftings}, Pacific Journal of Mathematics, Vol. 44, No. 2, (1973).
	
		\bibitem{PC} P. Cembranos, {\em $C(K,E)$ contains a complemented copy of $c_0$}, Proceedings of the American Mathematical Society, Vol. 91, No. 4 (Aug., 1984), pp.556-558
		
	  \bibitem{CW} E.W. Cheney and D.E. Wulbert, {\em The existence and unicity of best approximations}
	Math. Scand. \textbf{24} (1969), 113-140.

	\bibitem{D} F. Deutsch, {\em Linear selections for the metric projection,}
J. Functional Analysis  \textbf{49} (1982), 269-292.

	
	\bibitem{hww} P. Harmand, D. Werner and W. Werner, {\em M-ideals in Banach spaces and Banach algebras}, Lecture Notes in Mathematics, \textbf{1547} (1993) Springer-Verlag, Berlin.
	
\bibitem{hk} R. B. Holmes and B. R. Kripke, {\em Smoothness of approximation}, Michigan Math. J. \textbf{15} (1968), 225-248.

	\bibitem {LT} J. Lindenstrauss and L. Tzafriri, {\em Classical  Banach spaces I and II}, Springer Classical Mathematics. Springer-Verlag London, Ltd., London (1996).
	
		\bibitem{FJM }F. J. Murray, {\em On complementary manifolds and projections in spaces $L_p$ and $\ell_p$,} Trans. Amer. Math. Soc. \textbf{41:1} (1937),  138–152. 
	


	\bibitem {R} R. Ryan, {\em Introduction to tensor products of Banach spaces}, Springer Monographs in Mathematics. Springer-Verlag London, Ltd., London (2002),  225 pp.


	\bibitem{DY} D. Yost, {\em Best approximation operators in functional analysis}. Miniconference on nonlinear analysis, (1984) 249-270, Canberra AUS, 1984. 
\end{thebibliography}
\end{document}